\documentclass{amsart}
\usepackage{amsfonts}
\usepackage{amsthm}
\usepackage{amsmath}
\usepackage[mathscr]{eucal}
\numberwithin{equation}{section}
\usepackage{nicefrac}
\usepackage{pgfplots, pgfplotstable}
\usepackage[font=large]{subcaption}

\newcommand{\uproman}[1]{\uppercase\expandafter{\romannumeral#1}}

\theoremstyle{plain}
\newtheorem{Th}{Theorem}[section]
\newtheorem{Lemma}[Th]{Lemma}

\theoremstyle{definition}
\newtheorem{Def}[Th]{Definition}
\newtheorem{Rem}[Th]{Remark}
\newtheorem{Ex}[Th]{Example}

\DeclareMathOperator{\Span}{span}
\DeclareMathOperator{\Identity}{I}
\DeclareMathOperator{\DeltaAmpl}{dn}

\newcommand{\scp}[3]{\langle #1,#2\rangle_{#3}}
\newcommand{\IntScp}[4]{\scp{#1}{#2}{\textrm{#3}^{#4}(\Vnull,\Veins)}}

\newcommand{\IntScpVr}[4]{\scp{#1}{#2}{#3_r^{#4}}}
\newcommand{\IntScpVrDual}[4]{\scp{#1}{#2}{#3_{*,r}^{#4}}}
\newcommand{\minim}[1]{{\footnotesize{\textsc{v}_{#1}}}}
\newcommand{\minimdual}[1]{{\LARGE\textsc{v}_{#1}}}
\newcommand{\Vh}{\textsc{V}_h}
\newcommand{\Vr}{\mathcal{V}_r}
\newcommand{\Vnull}{\mathscr{V}_0}
\newcommand{\Hilbert}[1]{\mathscr{V}_{#1}}
\newcommand{\Veins}{\mathscr{V}_1}
\newcommand{\IntNorm}[3]{\|#1\|_{\textrm{#2}^{#3}(\Vnull,\Veins)}}
\newcommand{\IntNormGeneral}[7]{\|#1\|_{\textrm{#2}^{#3}((#4,\Norm{\cdot}{#5}),(#6,\Norm{\cdot}{#7}))}}
\newcommand{\IntNormDual}[3]{\|#1\|_{\textrm{#2}^{#3}(\Hilbert{-1},\Vnull)}}

\newcommand{\IntNormRB}[3]{\|#1\|_{#2_r^{#3}}}
\newcommand{\IntNormRBDual}[3]{\|#1\|_{{#2}_{*,r}^{#3}}}

\newcommand{\TheIntSpace}{[\Vnull,\Veins]_{s}}
\newcommand{\R}{\mathbb{R}}
\newcommand{\N}{\mathbb{N}}
\newcommand{\Op}{\mathcal{L}}
\newcommand{\IntCouple}{(\Vnull,\Veins)}
\newcommand{\IntOp}[2]{\Op_{\textsc{#1}^{#2}(\Vnull,\Veins)}}
\newcommand{\IntOpGeneral}[6]{\Op_{\textsc{#1}^{#2}((#3,#4),(#5,#6))}}
\newcommand{\IntOpDual}[2]{\Op_{\textsc{#1}^{#2}(\Hilbert{-1},\Vnull)}}
\newcommand{\IntOpRBDual}[2]{\Op_{#1_{*,r}}^{#2}}
\newcommand{\IntOpRB}[2]{\Op_{#1_r}^{#2}}

\newcommand{\IntMatOpVr}[2]{L_{#1_{r}}^{#2}}
\newcommand{\IntMatOpVrDual}[2]{L_{#1_{*,r}}^{#2}}

\newcommand{\RBSolutionDual}[1]{u_{r}^*(#1)}
\newcommand{\Krm}{\mathrm{K}}
\newcommand{\Norm}[2]{\|#1\|_{#2}}
\newcommand{\VecVh}[1]{\underline{#1}}
\newcommand{\VecVr}[1]{\underline{\underline{#1}}}
\newcommand{\Family}[3]{(#1)_{#2}^{#3}}
\newcommand{\K}{\mathrm{K}}
\newcommand{\Sum}[2]{\sum\limits_{#1}^{#2}}
\newcommand{\Ainvr}[1]{A_{*,r}^{#1}}

\newcommand{\Riesz}{\mathcal{R}_{1}}
\newcommand{\RieszV}{\mathcal{R}_{\Hilbert{}}}
\newcommand{\RieszVh}{\mathcal{R}}

\begin{document}
	\title{A Reduced Basis Method For Fractional Diffusion Operators \uproman{2}}

	\author{Tobias Danczul}
	\author{Joachim Sch\"oberl}
	\address{TU Wien \\ Institute for Analysis and Scientific Computing \\
		Wiedner Hauptstrasse 8-10, 1040 Wien, Austria} 
	\email{tobias.danczul@tuwien.ac.at}
	\email{joachim.schoeberl@tuwien.ac.at}

	\keywords{Fractional Laplace, Reduced Basis Method, Interpolation spaces, Zolotar\"ev points, Finite Element Method} 
	
	\subjclass[2010]{Primary  46B70, 65N12, 65N15, 65N30, 35J15; Secondary 65N25, 35P10}

	\begin{abstract} 
	We present a novel numerical scheme to approximate the solution map $s\mapsto u(s) := \Op^{-s}f$ to partial differential equations involving fractional elliptic operators. Reinterpreting $\Op^{-s}$ as interpolation operator allows us to derive an integral representation of $u(s)$ which includes solutions to parametrized reaction-diffusion problems. We propose a reduced basis strategy on top of a finite element method to approximate its integrand. Unlike prior works, we deduce the choice of snapshots for the reduced basis procedure analytically. Avoiding further discretization, the integral is interpreted in a spectral setting to evaluate the surrogate directly. Its computation boils down to a matrix approximation $L$ of the operator whose inverse is projected to a low-dimensional space, where explicit diagonalization is feasible. The universal character of the underlying $s$-independent reduced space allows the approximation of $(u(s))_{s\in(0,1)}$ in its entirety. We prove exponential convergence rates and confirm the analysis with a variety of numerical examples.
	
	Further improvements are proposed in the second part of this investigation to avoid inversion of $L$. Instead, we directly project the matrix to the reduced space, where its negative fractional power is evaluated. A numerical comparison with the predecessor highlights its competitive performance.
	\end{abstract}
	
	\maketitle
	
	\section{Introduction}
	The past twenty years have seen rapid advances in the field of fractional calculus. A significant amount of recent publications have been recognizing its potential to enhance their mathematical model in a sophisticated way. The range of application is broadly scattered, see e.g., \cite{ApplPopulDyn}, \cite{ApplApplebaum}, \cite{ApplImageProc}, \cite{ApplMaterialSc}, \cite{ApplPorousMediaFlow}, \cite{SpecApprFracPDES}, \cite{ApplicationBloodFlow}, \cite{ApplicationIschemia}, and \cite{sMapAppl2}. The arising interest in fractional powers of differential operators evokes the demand for robust and reliable numerical schemes. The challenge of their design is in particular a matter of efficiency. A vast body of literature has been published in order to tackle these difficulties, first and foremost by means of the fractional Laplace model problem
	\begin{align}
	\begin{aligned}\label{SecInt:ModelProblem}
			(-\Delta)^su &= f, \quad&&\text{in }\Omega, \\
			u &= 0, &&\text{on }\partial\Omega,
	\end{aligned}
	\end{align}
	for a bounded Lipschitz domain $\Omega\subset\R^d$, $d =1,2,3$, $f\in L_2(\Omega)$, and $s\in(0,1)$. We refer to \cite{AdaptivityFaustmann}, \cite{MelenkHMatrix}, \cite{AinsworthAdaptivFEMFracLapl}, \cite{MultilevMethFracDiff}, \cite{AinsworthEfficFEMFracLapl}, and \cite{TimeStepFracLaplace} to name a few of them. The literature provides a large variety of non-equivalent definitions of $(-\Delta)^s$. A comprehensive survey over its versatile definitions as well as the comparison of both existing and newly proposed numerical schemes is performed in \cite{FracLaplaceReview}, see also \cite{RatFracLaplace}, \cite{ReviewFracLaplace}, and \cite{EquivDef}. In this paper, we are concerned with fractional elliptic operators defined via spectral expansion.
	
	A conceptually straightforward approach is the accurate but expensive discrete eigenfunction method, as it is referred to in e.g., \cite{FracLaplaceReview}. It relies on a matrix approximation of the desired operator, whose $s^{th}$ power is evaluated, see also \cite{MTT1} and \cite{MTT2}. Due to its considerable computational effort, this approach is justified only if the problem-size is small.
	
	The algorithm developed in \cite{Pasciak} and later improved in \cite{SincQuadImproved} relies on the Dunford-Taylor integral representation of $\Op^{-s}$ and comes in two stages. First, the integral is approximated by a quadrature scheme especially tailored for problems of this type. Evaluation of the integrand in its quadrature nodes amounts to the computation of $\minim{}(t) := (\Identity-t^2\Op)^{-1}f$, $t\in\R^+$, which is why $\minim{}(t)$ is replaced by a finite element approximation in the second step. Exponential decay of the error in the number of quadrature nodes is shown. Further improvements are discussed in \cite{RBMBonito} and \cite{RBMKatoa}, where a third layer of approximation is added in the form of a reduced basis method. The choice of the underlying reduced space relies on a weak greedy algorithm. In \cite{RBMBonito}, the space is chosen independently of $s\in[s_\text{min},s_\text{max}]$ with $0<s_{\text{min}} \leq s_{\text{max}}<1$. Utilizing similar techniques as in \cite{RBMForFracDiff}, the authors of \cite{RBMBonito} prove exponential convergence rates for their algorithm. Comparable results are observed experimentally in \cite{RBMKatoa} for a different quadrature with computable upper bounds for the error. The fact that solutions to fractional differential equations are compressible, however, has been observed even earlier, see \cite{RBMForNonlocalOp} and \cite{CertifiedRB}.
	
	In \cite{PasciakBURA}, the so-called best uniform rational approximation of $t^{-s}$ is utilized on a spectral interval to approximate $\Op^{-s}$ efficiently, see also \cite{BURA}. Its evaluation also involves the computation of $\minim{}(t)$ at a selection of support points. As shown in \cite{RatFracLaplace}, a whole class of numerical methods admits the interpretation as rational approximation of a univariate function over a spectral interval.
	
	Of profound importance is the work of Caffarelli and Silvestre in \cite{HarmonicExt} and its variants \cite{ConConvEllipProb}, \cite{SqrRootLaplaceDirichlet}, \cite{RadialExtremalSol}, \cite{HarmExtGeneralized}. The idea is to rewrite the fractional differential equation as local degenerate integer order PDE on the semi-infinite cylinder $\mathcal{C} := \Omega\times\R^+$. A natural approach consists of a $d+1$ dimensional finite element method which takes advantage of the solution's rapid decay in the artificial direction, justifying truncation to a bounded domain of moderate size, see \cite{NochettoOtarolaSalgado}, \cite{MelenkTensorFEM}, \cite{MelenkRieder}, and \cite{hpTensorFEM}. The authors of \cite{AinsworthSpectralExt} avoided truncation of the cylinder by means of a spectral method in the extended direction. One of the pioneer model order reduction methods for the fractional Laplacian has been tailored in \cite{CertifiedRB}. The degenerate diffusion coefficient of the aforementioned boundary value problem is approximated in a way, such that the arising bilinear form is amenable to reduced basis technology. Exponential convergence is observed numerically.
	
	Adding to the difficulty of their non-local nature, the challenge of fractional differential equations is the fact that one is usually interested in the entire family of solutions $(u(s))_{s\in(0,1)}$. We refer to \cite{sMapAppl2} and \cite{sMapAppl1}, where the fractional order is utilized as additional tool to fit the mathematical model to the observed data. Furthermore, the authors of \cite{FamilySolMotivation} deal with a setting where $s = s(x)$ is a function of a spatial variable, which is why $s\mapsto u(s)$ can be seen as many-query problem. This is a challenge we particularly address in the development of our algorithm.
	
	We perform the approximation of $(u(s))_{s\in(0,1)}$ in two stages. Inspired by \cite{RBMForFracDiff}, we make use of the K-method \cite{RealMethoOfInt} to derive an integral representation of $u(s)$, whose evaluation requires the knowledge of $\minim{}(t)$ for all $t\in\R^+$. The first level of discretization consists of a standard finite element method that replaces the continuous solution $\minim{}(t)$ by its approximate counterpart $\minim{h}(t)$. The discrete integrand $\minim{h}(t)$ depends smoothly on $t$ and thus resides on a low-dimensional manifold. This justifies, in the second step, the usage of reduced basis technology, which seeks to approximate the entire family of solutions $(\minim{h}(t))_{t\in\R^+}$ by a finite selection $\Vr := \Span\{\minim{h}(t_0),...,\minim{h}(t_r)\}$ at a few strategical locations. The arising reduced basis integral is evaluated directly and does not require quadrature approximation. Its computation amounts to the assembly of $L$, a matrix approximation of the operator, whose inverse is projected to $\Vr$, where evaluations of its fractional powers can be determined directly. The construction of the reduced space is universal for all $s\in(0,1)$ and hence not restricted to proper subsets. Based on the analysis, we provide an optimal choice of sampling points for the reduced basis procedure. As opposed to prior works, they are given in closed form by means of the Zolotar\"ev points and, beside the knowledge of the extremal eigenvalues, do not require any further computations. Demanding slightly higher regularity assumptions than in \cite{RBMBonito}, we rigorously prove exponential convergence rates that only depend logarithmically on the condition number of the discrete operator.
	
	Inversion of $L$ becomes computationally prohibitive as the problem-size increases. To avoid this inconvenience, we present a second approach which directly projects the matrix in question to the small subspace where its negative fractional power is computed. Identical convergence rates as in the previous case are observed empirically.	We conclude this introduction by pointing out that the proposed method can be seen as model order reduction for the extension method without requiring truncation of the domain.
	
	The remainder of this paper is organized as follows. In Section \ref{SecInt} we provide a brief synopsis of the theory of interpolation spaces and their dual counterparts on which the proposed methodology is built upon. We design two different reduced basis procedures to approximate dual interpolation norms and solutions to \eqref{SecInt:ModelProblem} in Section \ref{SecRBM}, followed by a proof of its rapid convergence in Section \ref{SecAnalysis} for one of those algorithms. The purpose of Section \ref{SecNumericalEx} is a numerical comparison of the different approximations in order to highlight their performance by a variety of experiments. Finally, in the Appendix, we prove a couple of assertions whose validity has already been affirmed in the literature in some special cases, but not in the general framework that fits our setting. 
	
	\section{Notation and Preliminaries}\label{SecInt}
	Throughout what follows, by $a\preceq b$ we mean that $a\leq Cb$ for some generic constant $C\in\R^+$ which is independent of $a$, $b$, and the discretization parameters $h$ and $r$. The $s^{th}$ power of any symmetric matrix $A\in\R^{n\times n}$, $n\in\N$, is defined by diagonalization, i.e., $A^s := \Phi\Lambda^s\Phi^{-1}$, where $\Phi\in\R^{n\times n}$ denotes the matrix of column-wise arranged eigenvectors of $A$ and $\Lambda^s$ the involved diagonal matrix, containing the $s^{th}$ power of all corresponding eigenvalues. If $A$ is also positive definite, we set
	\begin{align*}
		\Norm{x}{A}^2 :=x^TAx, \qquad \scp{x}{y}{A} := x^TAy
	\end{align*} 
	for all $x,y\in\R^n$. Given a Banach space $(\Hilbert{},\Norm{\cdot}{\Hilbert{}})$, whose norm satisfies parallelogram law, we refer to the induced scalar product of $\Norm{\cdot}{\Hilbert{}}$ as the unique scalar product $\scp{\cdot}{\cdot}{\Hilbert{}}$ on $\Hilbert{}$ that satisfies $\Norm{\cdot}{\Hilbert{}} = \sqrt{\scp{\cdot}{\cdot}{\Hilbert{}}}$.
	Whenever referring to a Banach space $(\Hilbert{},\Norm{\cdot}{\Hilbert{}})$ as Hilbert space, we mean that its norm induces a scalar product $\scp{\cdot}{\cdot}{\Hilbert{}}$, such that $(\Hilbert{},\scp{\cdot}{\cdot}{\Hilbert{}})$ is a Hilbert space. By $\Hilbert{}'$ we denote the topological dual space of $\Hilbert{}$ henceforth, endowed with its natural norm 	
	\begin{align*}
		\Norm{F}{\Hilbert{}'} := \sup_{v\in\Hilbert{}}\frac{\scp{F}{v}{}}{\Norm{v}{\Hilbert{}}},
	\end{align*}
	where $\scp{\cdot}{\cdot}{}$ refers to the duality pairing. The norm comes from an inner product
	\begin{align*}
		\scp{F}{G}{\Hilbert{}'} = \scp{F}{\RieszV G}{},
	\end{align*}
	with $\RieszV:\Hilbert{}'\longrightarrow\Hilbert{}$ denoting the Riesz-isomorphism. The dual of a linear operator $T:\Veins\longrightarrow\Vnull$ between two Hilbert spaces is understood as the linear operator $T':\Vnull'\longrightarrow\Veins'$ uniquely defined by
	\begin{align*}
		\scp{F}{Tu}{} = \scp{T'F}{u}{}
	\end{align*}
	for all $u\in\Veins$ and $F\in\Vnull'$.
	
	\subsection{Hilbert space interpolation}\label{Sec:IntSpaces}
	On the basis of \cite{NonHomBVPandAppl}, \cite{Triebel}, \cite{IntToSob&IntSpac}, \cite{RealMethoOfInt}, \cite{IntSpaces}, and \cite{Bramble}, we briefly review the concept of Hilbert space interpolation. A pair of Hilbert spaces $(\Vnull,\Veins)$ with norms $\Norm{\cdot}{i}$, $i=0,1$, is admissible for space interpolation, if $\Veins\subset\Vnull$ is dense with compact embedding. In this case we call $(\Vnull,\Veins)$ interpolation couple. Fredholm theory ensures the existence of an orthonormal basis of eigenfunctions $(\varphi_k)_{k=1}^\infty$ of $\Vnull$ and a sequence of corresponding eigenvalues $\Family{\lambda_k^2}{k=1}{\infty}\subset\R^+$ with $\lambda_k^2\longrightarrow\infty$ as $k\to\infty$, such that
	\begin{align*}
		\llap{$\forall v\in \mathscr{V}_1:$\quad} \scp{\varphi_k}{v}{1} = \lambda_k^2\scp{\varphi_k}{v}{0}
	\end{align*}
	for all $k\in\mathbb{N}$. The {Hilbert interpolation norm then reads as
	\begin{align}\label{SecIntro:HilbertIntNorm}
		\|u\|_{\textrm{H}^s(\mathscr{V}_0,\mathscr{V}_1)}^2 := \sum_{k = 1}^{\infty}\lambda_k^{2s}\scp{u}{\varphi_k}{0}^2
	\end{align}
	for each $s\in(0,1)$. The arising interpolation space between $\Vnull$ and $\Veins$ is defined by
	\begin{align*}
		[\mathscr{V}_0,\mathscr{V}_1]_{s} := \{u\in \mathscr{V}_0: \IntNorm{u}{H}{s}<\infty\}
	\end{align*} 
	and turns out to be a Hilbert space itself. It is well-known that \eqref{SecIntro:HilbertIntNorm} admits an integral representation by means of the K-method. For each $u\in\Vnull$ we define the K-functional of $\IntCouple$ by
	\begin{align*}
		\Krm_{(\Vnull,\Veins)}^2(t;u) := \inf\limits_{v\in \Veins}\|u-v\|_0^2 + t^2\|v\|_1^2, \rlap{\qquad$t\in\R^+.$}
	\end{align*}	
	With this at hand, we introduce the K-norm as
	\begin{align*}
		\IntNorm{u}{K}{s}^2 := \int_0^\infty t^{-2s-1}\Krm_{\IntCouple}^2(t;u)\,dt.
	\end{align*}
	There holds
	\begin{align}\label{SecIntro:NormEquality}
		\IntNorm{\cdot}{H}{s} = C_s\IntNorm{\cdot}{K}{s}, \rlap{\qquad$C_s^2 := \frac{2\sin(\pi s)}{\pi},$}
	\end{align}
	see e.g., \cite{Bramble} for details.
	
	Here, and in all following, we define the Hilbert interpolation operator of $(\Vnull,\Veins)$ as the unique operator $\IntOp{H}{s}:\TheIntSpace\longrightarrow\Vnull$ that satisfies
	\begin{align*}
		\scp{v}{\IntOp{H}{s} u}{0} = \IntScp{v}{u}{H}{s}
	\end{align*}
	for all $v\in\Veins$, where $\IntScp{\cdot}{\cdot}{H}{s}$ labels the induced scalar product of $\IntNorm{\cdot}{H}{s}$. The K-interpolation operator $\IntOp{K}{s}$ is understood respectively. Due to \eqref{SecIntro:NormEquality}, there holds
	\begin{align}\label{SecIntro:IntOpEquality}
		\IntOp{H}{s} = C_s^2\IntOp{K}{s}.
	\end{align} 
	For any $u\in\TheIntSpace$ the action of the Hilbert interpolation operator is given by
	\begin{align*}
		\IntOp{H}{s}u = \Sum{k=1}{\infty}\lambda_k^{2s}\scp{\varphi_k}{u}{0}\varphi_k.
	\end{align*}
	An explicit representation of its latter counterpart requires further preparation.
	\begin{Lemma}\label{Lm:MinimDecomp}
		For each $u\in\Vnull$ and $t\in\R^+$ the minimizer $\minim{}(t)$ of $\K_{(\Vnull,\Veins)}^2\left(t;u\right)$ is the unique solution of
		\begin{align*}
			\scp{\minim{}(t)}{w}{0} + t^2\scp{\minim{}(t)}{w}{1} = \scp{u}{w}{0}
		\end{align*}
		for all $w\in\Veins$. It satisfies
		\begin{align*}
			\minim{}(t) = \Sum{k=1}{\infty}\frac{\scp{u}{\varphi_k}{0}}{1+t^2\lambda_k^2}\varphi_k.
		\end{align*}
	\end{Lemma}
	\begin{proof}
		The proof is done in complete analogy to \cite[Lemma 3.4]{RBMForFracDiff}, where the claim has been shown for the finite dimensional case.
	\end{proof}
	In the spirit of \cite{Pasciak}, we deduce an appealing integral representation of the K-interpolation operator. This has already been done for the finite-dimensional case in \cite{RBMForFracDiff}, but also applies to this general setting, as the following theorem shows.
	\begin{Th}\label{Th:IntRepr}
		Let $s\in(0,1)$, $u\in\TheIntSpace$, and $\minim{}(t)$ as in Lemma \ref{Lm:MinimDecomp}. Then there holds
		\begin{align*}
			\IntOp{K}{s}u = \int_{0}^{\infty}t^{-2s-1}\left(u - \minim{}(t)\right)dt.
		\end{align*}
	\end{Th}
	\begin{proof}
		See Appendix.
	\end{proof}
		
	The ambition of this investigation is to approximate the action of the inverse Hilbert interpolation operator
	\begin{align}\label{SecIntro:InverseIntOp}
		\IntOp{H}{s}^{-1}f = \Sum{k=1}{\infty}\lambda_k^{-2s}\scp{f}{\varphi_k}{0}\varphi_k
	\end{align}
	in $f$ and $s$.	To achieve this, we rewrite \eqref{SecIntro:InverseIntOp} as classical interpolation operator in a dual setting which allows us to exploit the reduced basis procedure from \cite{RBMForFracDiff}.

	\subsection{Interpolation of the dual spaces}\label{SecDualInterpolation}
	Let $\Hilbert{-1} := \Veins'$ denote the topological dual space of $\Veins$ and $\Riesz:\Hilbert{-1}\longrightarrow\Veins$ its Riesz-isomorphism, such that
	\begin{align*}
		\llap{$\forall v\in\Veins:$\quad}\scp{\Riesz F}{v}{1} = \scp{F}{v}{}
	\end{align*}
	for any  $F\in\Hilbert{-1}$. As usual, we identify each function $f\in\Vnull$ with the functional $F\in\Hilbert{-1}$ in terms of
	\begin{align}\label{SecIntro:Identification}
		\llap{$\forall v\in\Veins:$\quad}\scp{F}{v}{} = \scp{f}{v}{0}.
	\end{align}
	Along with this identification, we have that $\Veins\subset\Vnull\cong\Vnull'\subset\Hilbert{-1}$, which is why we do not distinguish between $\Vnull$ and $\Vnull'$ henceforth. As $\iota:\Veins\longrightarrow\Vnull$ denotes the embedding of $(\Vnull,\Veins)$, $\iota'$ is compact and injective with dense range, i.e., the pairing $(\Hilbert{-1},\Vnull)$ is an interpolation couple itself, see e.g., \cite[Chapter 41]{IntToSob&IntSpac}. As shown in \cite[Theorem 6.2]{NonHomBVPandAppl}, there further holds
	\begin{align*}
		\TheIntSpace' = [\Hilbert{-1},\Vnull]_{1-s},
	\end{align*}
	which is why we refer to $\IntNormDual{\cdot}{H}{s}$ and $[\Hilbert{-1},\Vnull]_{s}$ as dual interpolation norm and space. Consecutive elaborations provide an understanding how $\TheIntSpace$ and $[\Hilbert{-1},\Vnull]_{s}$ can be seen as single scale of interpolation spaces for $s\in(-1,1)$, see also \cite{Pasciak} for the case $\Vnull = L_2(\Omega)$ and $\Veins = H_0^1(\Omega)$. In view of \eqref{SecIntro:Identification}, we denote with $(\Psi_k,\widehat{\lambda}_k^2)_{k=1}^\infty$ the system of $\Hilbert{-1}$-orthonormal eigenfunctions, such that
	\begin{align*}
		\scp{\Psi_k}{f}{} = \widehat{\lambda}_k^2\scp{\Psi_k}{F}{-1}
	\end{align*}
	for all $F\in\Vnull$. Referring to $(\Phi_k)_{k=1}^\infty$ as the family of functionals associated to $(\varphi_k)_{k=1}^\infty$, these dual eigenpairs can be expressed in the following convenient manner.
	\begin{Th}\label{Thm:DualEigenpairs}
		For all $k\in\N$ there holds
		\begin{align*}
			(\Psi_k,\widehat{\lambda}_k^2) = (\lambda_k\Phi_k,\lambda_k^2).
		\end{align*}
	\end{Th}
	\begin{proof}
		The claimed identity has been shown for $\Vnull = L_2(\Omega)$ and $\Veins = H_0^1(\Omega)$ in \cite[Proposition 4.1]{Pasciak}. The proof also holds for this more general setting. For completeness, however, we carry out its details in the Appendix.
	\end{proof}
	\begin{Def}
		For all $s\in(0,1)$ we define the Hilbert extrapolation norm on $[\Hilbert{-1},\Vnull]_{1-s}$ by
		\begin{align}\label{SecIntro:NegNorm}
			\IntNorm{F}{H}{-s}^2 := \Sum{k=1}{\infty}\lambda_k^{-2s}\scp{F}{\varphi_k}{}^2
		\end{align}
		and denote with $\IntScp{\cdot}{\cdot}{H}{-s}$ its induced scalar product. The Hilbert extrapolation operator $\IntOp{H}{-s}:[\Hilbert{-1},\Vnull]_{1-s}\longrightarrow\Vnull$ is defined by
		\begin{align*}
			\scp{f}{\IntOp{H}{-s}G}{0} = \IntScp{F}{G}{H}{-s}, \rlap{\qquad$F\in\Vnull.$}
		\end{align*}
	\end{Def}
	The following theorem shows that \eqref{SecIntro:NegNorm} is essentially the dual interpolation norm and therefore well-defined. It summarizes the key ingredients of this section, on which our reduced basis approach is built upon.
	\begin{Th}\label{Th:DualvsNegative}
		Let $s\in(0,1)$ and $F\in[\Hilbert{-1},\Vnull]_{1-s}$. Then there holds
		\begin{align}\label{SecIntro:NormEqualityNew}
			\IntNorm{F}{\textsc{H}}{-s} = \IntNormDual{F}{\textsc{H}}{1-s} = C_{1-s}\IntNormDual{F}{\textsc{K}}{1-s}
		\end{align}
		and
		\begin{align}\label{SecIntro:EqualOperators}
			\IntOp{H}{-s}F = \Riesz\IntOpDual{H}{1-s}F  = C_{1-s}^2\Riesz\IntOpDual{K}{1-s}F.
		\end{align}
		Moreover, if $F\in\Vnull$, the operator actions listed above coincide with $\IntOp{H}{s}^{-1}f$.
	\end{Th}
	\begin{proof}
		See Appendix.
	\end{proof}
	\begin{Rem}
		We point out that Theorem \ref{Th:DualvsNegative} can be generalized in the sense of \cite[Corollary 2.7]{RBMForFracDiff}, such that the Hilbert interpolation operator admits the interpretation as Dirichlet-to-Neumann map in the extension setting from \cite{HarmonicExt}. 
	\end{Rem}
	Before we deal with a first discretization of \eqref{SecIntro:InverseIntOp}, another useful property is highlighted.
	\begin{Lemma}\label{Lm:EqualMinim}
		Let $f\in\Vnull$, $F$ as in \eqref{SecIntro:Identification}, $t\in\R^+$, $\minim{}(t)$ the minimizer of $\K_{(\Vnull,\Veins)}(t;f)$, and $\minimdual{}(t)$ the minimizer of $\K_{(\Hilbert{-1},\Vnull)}(t;F)$. Then $\minim{}(t)$ and $\minimdual{}(t)$ coincide, i.e.,
		\begin{align*}
			\llap{$\forall w\in\Veins:\quad$}\scp{\minimdual{}(t)}{w}{} = \scp{\minim{}(t)}{w}{0}.
		\end{align*}
	\end{Lemma}
	\begin{proof}
		Lemma \ref{Lm:MinimDecomp} applied to the interpolation couple $(\Hilbert{-1},\Vnull)$ reveals 
		\begin{align*}
			\minimdual{}(t) = \Sum{k=1}{\infty}\frac{\scp{F}{\Psi_k}{-1}}{1+t^2\lambda_k^2}\Psi_k = \Sum{k=1}{\infty}\lambda_k^{-2}\frac{\scp{F}{\psi_k}{}}{1+t^2\lambda_k^2}\Psi_k = \Sum{k=1}{\infty}\frac{\scp{f}{\varphi_k}{0}}{1+t^2\lambda_k^2}\Phi_k,
		\end{align*}
		affirming the claim.
	\end{proof}
	\subsection{Finite Element Approximation}
	We are concerned with the discretization of the problem: Find $u\in\TheIntSpace$, such that
	\begin{align}\label{SecIntro:GeneralPDE}
		\IntOp{H}{s}u = f
	\end{align}
	with $f\in\Vnull$. For concreteness, let $\Vh\subset\Veins$ denote a finite element space which is spanned by a $\Vnull-$orthonormal basis of eigenfunctions $(\varphi_{h,k})_{k=1}^N$, satisfying
	\begin{align}\label{SecIntr:DiscreteEigenpairs}
		\llap{$\forall v_h\in\Vh:$\quad}\scp{\varphi_{h,k}}{v_h}{1} = \lambda_{h,k}^2\scp{\varphi_{h,k}}{v_h}{0}
	\end{align}
	for any $k = 1,...,N$. For convenience, we set
	\begin{align*}
		 \Op_{H}^s := \Op_{\textsc{H}^s\left((\Vh,\Norm{\cdot}{0}),(\Vh,\Norm{\cdot}{1})\right)}
	\end{align*} 
	and denote by $\Op_{H}^{-s}$ its inverse. The discrete eigenfunction method relies on the approximation
	\begin{align}\label{SecIntro:DEM}
		u(s)\approx u_h(s) := \Op_{H}^{-s}\pi_hf = \sum_{k = 1}^N\lambda_{h,k}^{-2s}\scp{\varphi_{h,k}}{f}{0}\varphi_{h,k},
	\end{align}
	where $\pi_hf$ refers to the $\Vnull$-orthogonal projection of $f$ onto $\Vh$. For a fixed basis $(b_{h,k})_{k=1}^N$ of $\Vh$ let $M, A\in\R^{N\times N}$ denote the mass- and stiffness-matrix arising from finite element discretization in the sense of
	\begin{align*}
		M_{ji} = \scp{b_{h,i}}{b_{h,j}}{0},\qquad A_{ji} = \scp{b_{h,i}}{b_{h,j}}{1}.
	\end{align*} 
	Moreover, for any $v_h\in\Vh $ we write $\VecVh{v_h}\in\R^N$ to refer to its collection of degrees of freedoms, such that
	\begin{align*}
		v_h = \sum\limits_{k=1}^N (\VecVh{v_h})_k b_{h,k}.
	\end{align*} 
	With this at hand, \eqref{SecIntro:DEM} can be computed by means of
	\begin{align*}
		\VecVh{u_h}(s) = L^{-s}\VecVh{\pi_hf},\rlap{\qquad $L:=M^{-1}A.$}
	\end{align*} 
	As shown in \cite[Theorem 4.3]{Pasciak}, $u_h(s)$ serves as accurate approximation to $u(s)$ and satisfies quasi-optimal convergence rates. Having complexity of $\mathcal{O}(N^3)$, however, this approach is only justified if the problem-size is moderate. To circumvent this restriction, we present the construction of a reduced basis surrogate $u_{h,r}(s)\approx u_h(s)$, whose evaluation is performed with significant acceleration.
	
	Assumed that the mesh size $h$ is sufficiently small, we regard $u_h(s)$ as our underlying truth solution. For clarity in exposition, we therefore set
	\begin{align*}
		\Norm{F}{-1} = \Norm{F}{(\Vh,\Norm{\cdot}{1})'} = \sup_{v\in\Vh}\frac{\scp{F}{v}{}}{\Norm{v}{1}}, \rlap{\qquad $F\in\Vh',$}
	\end{align*}
	and neglect the subscript $h$ for all finite element function henceforth, such that $(\varphi_k,\lambda_k^2)_{k=1}^N$ labels the discrete eigenpairs in \eqref{SecIntr:DiscreteEigenpairs} from now on. In short, we do not incorporate the continuous level any further. We point out, however, that consecutive estimates can be traced back to the continuous solution by means of triangle inequality and the results form \cite{Pasciak}. 
	
	\section{Reduced Basis Approximation}\label{SecRBM}
	The goal of this section is to describe an accurate and yet efficient approximation of the manifold of finite element solutions $(u(s))_{s\in(0,1)}$ arising from \eqref{SecIntro:GeneralPDE}. We present two different algorithms to approximate the map $(f,s)\mapsto u(s)$ on $\Vh\times(0,1)$ at downsized computational effort. The first one interprets the inverse operator as classical interpolation operator in a dual setting. This allows us to exploit the reduced basis procedure developed in \cite{RBMForFracDiff} to approximate $u(s)$ at exponential rates. 
	The latter algorithm is motivated by Theorem \ref{Th:DualvsNegative} and aims to approximate solutions to \eqref{SecIntro:GeneralPDE} in the sense of the Hilbert extrapolation operator.
	
	\subsection{Dual approximation}\label{SecNormAppr}
	We are interested in approximations of the discrete dual interpolation norms
	\begin{align*}
		\Norm{f}{H^{-s}} &:= \IntNormGeneral{f}{H}{1-s}{\Vh'}{-1}{\Vh}{0},\\
		\Norm{f}{K^{-s}} &:= \IntNormGeneral{f}{K}{1-s}{\Vh'}{-1}{\Vh}{0}, 
	\end{align*}
	together with their inherently related interpolation operators. Initiated by \cite{RBMForFracDiff}, we define for each $f\in\Vh$ its reduced basis approximations as follows.
	\begin{Def}[Dual reduced basis algorithm]
		For each $t\in\R^+$ let $\minim{N}(t)\in\Vh$ denote the unique solution of
		\begin{align}\label{SecRBM:EulerLagrangeLSE}
			\scp{\minim{N}(t)}{w}{0} + t^2\scp{\minim{N}(t)}{w}{1} = \scp{f}{w}{0}
		\end{align}
		for all $w\in\Vh$. Given some snapshots $0 = t_0 < t_1 < ... < t_r$, specified in Section \ref{SecAnalysis}, we introduce the reduced space
		\begin{align}\label{Def:ReducedSpace}
			\Vr :=\Span\{\minim{N}(t_0),...,\minim{N}(t_r)\}\subset\Vh.
		\end{align}
		The dual reduced basis interpolation norms on $\Vr$ are defined by either of the two equivalent definitions
		\begin{alignat*}{2}
			&\IntNormRBDual{f}{H}{-s} &&:= \|f\|_{\textrm{H}^{1-s}((\Vr',\Norm{\cdot}{-1}),(\Vr,\Norm{\cdot}{0}))},\\
			&\IntNormRBDual{f}{K}{-s} &&:= \IntNormGeneral{f}{\textrm{K}}{1-s}{\Vr'}{-1}{\Vr}{0}.
		\end{alignat*}
		Moreover, we define the dual reduced basis operators as
		\begin{align*}
			\IntOpRBDual{H}{-s} &:= \RieszVh\IntOpRBDual{H^{1-s}}{} := \RieszVh\IntOpGeneral{H}{1-s}{\Vr'}{\Norm{\cdot}{-1}}{\Vr}{\Norm{\cdot}{0}},\\
			\IntOpRBDual{K}{-s} &:= \RieszVh\IntOpRBDual{K^{1-s}}{} := \RieszVh\IntOpGeneral{K}{1-s}{\Vr'}{\Norm{\cdot}{-1}}{\Vr}{\Norm{\cdot}{0}},
		\end{align*}
		where $\RieszVh$ denotes the Riesz-isomorphism of $(\Vh,\Norm{\cdot}{1})$. The dual reduced basis approximation of $u(s)$ is defined by
		\begin{align*}
			\RBSolutionDual{s} := \IntOpRBDual{H}{-s}(f).
		\end{align*}
	\end{Def}
	\begin{Rem}
		The choice of $\Vr$ is motivated by Lemma \ref{Lm:MinimDecomp} and \ref{Lm:EqualMinim}. Due to $\minim{N}(t_0) = \minim{N}(0) = f$, we have that $f\in\Vr$. As shown in \cite[Lemma 3.5]{RBMForFracDiff}, the set $\{\minim{N}(t_0),...,\minim{N}(t_r)\}$ is linearly independent as long as $r+1$ is less or equal than the number of excitations of $f$.
	\end{Rem}
	As shown in \cite[Theorem 3.6 \& Theorem 4.4]{RBMForFracDiff}, a lucky break down might occur, if the number of excitations of $f$ is rather small.
	\begin{Th}\label{Th:ExactRBNorm}
		Let $|\{\lambda_k^2\in\{\lambda_1^2,...,\lambda_N^2\}:\scp{f}{\varphi_k}{0}\neq 0\}|=:m\in\N$. If $r+1\ge m$, then the dual reduced basis interpolation norms coincide with the discrete dual interpolation norms, respectively. Moreover, there holds
		\begin{align*}
			u(s) = u_r^*(s).
		\end{align*}
	\end{Th}
	\subsubsection{Computational aspects}
	Implementation of the dual algorithm consists of the typical ingredients that are used in a modern reduced basis algorithm and can be found in the literature, see e.g., \cite{RBMRef4}, \cite{RBMRef2}, \cite{RBMRef3}, and \cite{KunischVolkwein}. For completeness, we give a brief synopsis of the relevant aspects involved. We define columnise the matrix
	\begin{align*}
		\widehat{V}_r := [\VecVh{\minim{N}}(t_0),...,\VecVh{\minim{N}}(t_r)]\in\R^{N\times(r+1)}
	\end{align*}
	as collection of basis vectors of $\Vr$. In favour of numerical stability, let $V_r \in \R^{N\times(r+1)}$ denote the unique matrix that arises from Gram-Schmidt orthonormalization chronologically applied	to the columns of $\widehat{V}_r$ with respect to the scalar product $\scp{\cdot}{\cdot}{M}$. With this at hand, we define the projected dual matrix
	\begin{align*}
		\Ainvr{} := V_r^TMA^{-1}MV_r\in\R^{(r+1)\times(r+1)}
	\end{align*}
	and note that for any $v_r,w_r\in\Vr$ there holds by construction
	\begin{align*}
		\scp{v_r}{w_r}{0} = \VecVr{v_r}^T\VecVr{w_r},\qquad \scp{v_r}{w_r}{-1} = \VecVr{v_r}^T\Ainvr{}\VecVr{w_r},
	\end{align*}
	where
	\begin{align*}
		\VecVr{u_r} := V_r^TM\VecVh{u_r}, \qquad\text{such that}\qquad \VecVh{u_r} = V_r\VecVr{u_r},
	\end{align*}
	for any $u_r\in\Vr$. We are now on position to explicitly compute $\RBSolutionDual{s}$ by means of the involved matrix representation of $\IntOpRBDual{H}{-s}$, i.e., the matrix $\IntMatOpVrDual{H}{-s}\in\R^{N\times N}$, such that
	\begin{align*}
	\scp{v}{\IntOpRBDual{H}{-s}(f)}{0} = \scp{v}{\IntMatOpVrDual{H}{-s}\VecVh{f}}{M}, \rlap{\qquad$v\in\Vh.$}
	\end{align*}
	\begin{Th}
		Let $f\in\Vh$. Then there holds
		\begin{align}\label{SecRBM:DualNormRepr}
			\IntNormRBDual{f}{H}{-s} = \Norm{\VecVr{f}}{\Ainvr{s}}.
		\end{align}
		The induced scalar product $\IntScpVrDual{\cdot}{\cdot}{H}{-s}$ on $(\Vr, \IntNormRBDual{\cdot}{H}{-s})$ satisfies
		\begin{align}\label{SecRBM:DualScpRepr}
			\IntScpVrDual{v_r}{w_r}{H}{-s} = \VecVr{v_r}^T\Ainvr{s}\VecVr{w_r}
		\end{align}
		for all $v_r,w_r\in\Vr$. The matrix representation of $\IntOpRBDual{H}{-s}$ is given by
		\begin{align}\label{SecRBM:DualMatRepr}
			\IntMatOpVrDual{H}{-s} = A^{-1}MV_r\Ainvr{s-1}V_r^TM.
		\end{align}
	\end{Th}
	\begin{proof}
		Let $(\psi_j,\mu_j^2)_{j=0}^r\subset\Vr\times\R^+$ denote the system of eigenpairs arising from the interpolation couple $((\Vr',\Norm{\cdot}{-1}),(\Vr,\Norm{\cdot}{0}))$, such that
		\begin{align*}
			\scp{\psi_j}{v_r}{0} = \mu_j^2\scp{\psi_j}{v_r}{-1},\rlap{\qquad$v_r\in\Vr.$}
		\end{align*}
		In terms of degrees of freedoms this reads as
		\begin{align*}
			\VecVr{\psi_j} = \mu_j^2\Ainvr{}\VecVr{\psi_j},\qquad\text{or equivalently,}\qquad \Ainvr{}\VecVr{\psi_j} = \mu_j^{-2}\VecVr{\psi_j}.
		\end{align*}
		Theorem \ref{Th:DualvsNegative} with $\phi_j := \mu_j^{-1}\psi_j$ reveals
		\begin{align*}
			\IntNormRBDual{f}{H}{-s}^2 = \Sum{j=0}{r}\mu_j^{2-2s}\scp{f}{\psi_j}{-1}^2 = \Sum{j=0}{r}\mu_j^{-2s}\scp{f}{\phi_j}{0}^2 = \Sum{j=0}{r}\mu_j^{-2s}\scp{\VecVr{f}}{\VecVr{\phi_j}}{I_r}^2 = \Norm{\VecVr{f}}{\Ainvr{s}}^2,
		\end{align*}
		where $I_r\in\R^{(r+1)\times(r+1)}$ denotes the unit matrix. This proves \eqref{SecRBM:DualNormRepr} and \eqref{SecRBM:DualScpRepr}. To confirm \eqref{SecRBM:DualMatRepr}, we deduce
		\begin{align*}
			\IntOpRBDual{H^{1-s}}{}(f) = \sum_{j=0}^r\mu_j^{2-2s}\scp{\psi_j}{f}{-1}\psi_j = \sum_{j=0}^{r}\mu_j^{2-2s}\scp{\phi_j}{f}{0}\phi_j,
		\end{align*}
		such that
		\begin{align*}
			\scp{w}{\IntOpRBDual{H^{1-s}}{}(f)}{-1} = \scp{\VecVh{w}}{V_r\Ainvr{s-1}\VecVr{f}}{MA^{-1}M},
		\end{align*}
		for all $w\in\Vh$, concluding the proof.
	\end{proof}
	\begin{Rem}
		By construction of $V_r$, there holds $\VecVr{f} = \beta V_r\VecVr{e_1}$, where $\VecVr{e_1}\in\R^{r+1}$ refers to the first unit vector and $\beta = \Norm{f}{0}$.  
	\end{Rem}

	The present algorithm can be seen as prototype of a modern reduced basis algorithm that substantially alleviates the costs of direct computations. The dominant contributions to the computational effort come from $r$ potentially expensive finite element approximations to standard reaction-diffusion problems of the original problem-size as well as the inversion of $A$. The latter entails considerable difficulties, if the problem-size is too large to compute $A^{-1}$ directly in the course of a one-time investment. In this case, one might resort to approximate the action $\VecVh{v}\mapsto A^{-1}\VecVh{v}$ iteratively, amounting to $r+2$ additional solves that come from the assembly of $\Ainvr{}$ and the application of the discrete Riesz-isomorphism $\mathcal{R}$. These costs might be diminished \cite{FaustmannHMatrix}, but not entirely eliminated.
 	
 	\subsection{Extrapolation-based approximation}
 	In the remainder of this section, we present a different, but conceptually similar, algorithm that resolves the inconveniences listed above, while at the same time, allows the observation of the competitive convergence rates that are proven for the first scheme in Section \ref{SecAnalysis}. Motivated by Theorem \ref{Th:DualvsNegative}, we make use of the orthonormal system obtained by $((\Vr,\Norm{\cdot}{0}),(\Vr, \Norm{\cdot}{1}))$ and extrapolate its interpolation norm for $s\in(-1,0)$.
 	\begin{Def}[Extrapolation-based reduced basis algorithm]\label{Def:RBM2} 		
 		We introduce the reduced basis extrapolation norms by
 		\begin{align*}
 			\IntNormRB{f}{H}{-s} &:= \IntNormGeneral{f}{H}{-s}{\Vr}{0}{\Vr}{1}, \\
 			\IntNormRB{f}{K}{-s} &:= \IntNormGeneral{f}{K}{-s}{\Vr}{0}{\Vr}{1}.
 		\end{align*}
 		The reduced basis extrapolation operators are defined as
 		\begin{align*}
 			\IntOpRB{H}{-s} &:= \IntOpGeneral{H}{-s}{\Vr}{\Norm{\cdot}{0}}{\Vr}{\Norm{\cdot}{1}},\\
 			\IntOpRB{K}{-s} &:= \IntOpGeneral{K}{-s}{\Vr}{\Norm{\cdot}{0}}{\Vr}{\Norm{\cdot}{1}}.
 		\end{align*} 
 		The extrapolation-based reduced basis approximation of $u(s)$ is defined by
 		\begin{align*}
 			u_r(s) := \IntOpRB{H}{-s}(f).
 		\end{align*}
 	\end{Def}
 	\begin{Rem}
 		We point out that the extrapolation-based reduced basis algorithm is distinct from its dual counterpart, i.e., $\IntNormRB{f}{H}{-s}\neq\IntNormRBDual{f}{H}{1-s}$. This is consistent with Theorem \ref{Th:DualvsNegative}, as $\Norm{\cdot}{(\Vr,\Norm{\cdot}{1})'} \neq\Norm{\cdot}{-1}$.
 	\end{Rem}
 	\begin{Th}
 		Let $|\{\lambda_k^2\in\{\lambda_1^2,...,\lambda_N^2\}:\scp{f}{\varphi_k}{0}\neq 0\}| =: m\in\N$. If $r+1\geq m$, then the reduced basis extrapolation norms coincide with the discrete dual interpolation norms, respectively. Moreover, there holds
 		\begin{align*}
 			u(s) = u_r(s).
 		\end{align*}
 	\end{Th} 
 	\begin{proof}
 		In the proof of \cite[Theorem 3.6]{RBMForFracDiff} it has been shown that the eigenpairs of $((\Vr,\Norm{\cdot}{0}),(\Vr,\Norm{\cdot}{1}))$ coincide with $\left(\frac{f^{i_j}}{\Norm{f^{i_j}}{0}},\lambda_{i_j}^2\right)$, 	if $r+1\ge m$. Here, $f^{i_j}$ refers to the $0$-orthonormal projection of $f$ in the eigenspace corresponding to the eigenvalue $\lambda_{i_j}^2$ and $\{i_0,...,i_{m-1}\}\subset\{1,...,N\}$ is chosen such that 
 		\begin{align*}
 			\{\lambda_{i_0}^2,...,\lambda_{i_{m-1}}^2\}= \{\lambda_k^2\in\{\lambda_1^2,...,\lambda_N^2\}:\scp{f}{\varphi_k}{0}\neq 0\}.
 		\end{align*}
 		This reveals
 		\begin{align*}
 			\IntNormRB{f}{H}{-s}^2 = \sum\limits_{j=0}^{m-1}\lambda_{i_j}^{-2s}\frac{\scp{f}{f^{i_j}}{0}^2}{\Norm{f^{i_j}}{0}^2}
 			= \sum\limits_{j=0}^{m-1}\lambda_{i_j}^{-2s}\Norm{f^{i_j}}{0}^2
 			= \Norm{f}{H^{-s}}^2,
 		\end{align*}
 		proving the first identity. Analog computations confirm the latter and finish the proof.
 	\end{proof}
 	Implementation of the reduced basis extrapolation operator relies on the projected stiffness matrix
 	\begin{align*}
 		A_r : = V_r^TAV_r\in\R^{(r+1)\times(r+1)},
 	\end{align*}
 	being of further interest in the subsequent theorem.
 	\begin{Th}\label{Th:ComutationExtrapolationAlg}
 		Let $f\in \Vh$. Then there holds
 		\begin{align*}
 			\IntNormRB{f}{H}{-s} = \VecVr{f}^TA_r^{-s}\VecVr{f}.
 		\end{align*}
 		The induced scalar product $\IntScpVr{\cdot}{\cdot}{H}{-s}$ on $(\Vr,\IntNormRB{\cdot}{H}{-s})$ satisfies
 		\begin{align*}
 			\IntScpVr{v_r}{w_r}{H}{-s} = \VecVr{v_r}^TA_r^{-s}\VecVr{w_r}
 		\end{align*}
 		for all $v_r,w_r\in\Vr$. The matrix representation of $\IntOpRB{H}{-s}$ is given by
 		\begin{align*}
 			\IntMatOpVr{H}{-s} = V_rA_r^{-s}V_r^TM.
 		\end{align*}
 	\end{Th}
 	\begin{proof}
 		Let $(\phi_j,\mu_j^2)_{j=0}^r$ denote the eigenpairs of $((\Vr,\Norm{\cdot}{0}),(\Vr,\Norm{\cdot}{1}))$. Since
 		\begin{align*}
 			A_r\VecVr{\phi_j} = \mu_j^2\VecVr{\phi_j},
 		\end{align*}
 		there holds
 		\begin{align*}
 			\Norm{f}{H_r^{-s}}^2 = \Sum{j=0}{r}\mu_j^{-2s}\scp{f}{\phi_j}{0}^2 = \Sum{j=0}{r}\mu_j^{-2s}\scp{\VecVr{f}}{\VecVr{\phi_j}}{I_r}^2 = \Norm{\VecVr{f}}{A_r^{-s}}^2,
 		\end{align*}
 		where $I_r\in\R^{(r+1)\times(r+1)}$ denotes the unit matrix. This validates the first part of the claim. Its remainder follows analogously.
 	\end{proof}
 	\begin{Rem}
 		Theorem \ref{Th:ComutationExtrapolationAlg} emphasizes the discrepancy between $\IntNormRB{\cdot}{H}{-s}$ and $\IntNormRBDual{\cdot}{H}{-s}$ in the sense of how the approximation of $(\varphi_k)_{k=1}^N$ is designed. The reduced basis extrapolation norm projects the stiffness matrix to the low-dimensional space, followed by computing its negative fractional power. This stands in contrast to the latter norm, which requires inversion of $A$ first and then performs the projection to $\Vr$. In short, the difference between both norms is based on the fact that inversion and projection of the corresponding matrices do not commute. However, as $A_r$ turns out to capture the characteristics of the eigenproblem \eqref{SecIntr:DiscreteEigenpairs} very well, it is to be expected that $A_r^{-1}$ provides a decent approximation for the dual setting. This understanding matches our observations in Section \ref{SecNumericalEx}. 
 	\end{Rem}
 	Definition \ref{Def:RBM2} is very attractive as it essentially breaks down to $r$ shifted problems of type \eqref{SecRBM:EulerLagrangeLSE} that can be solved in parallel. Efficient iterative solution algorithms are available whose convergence rates are independent of the underlying mesh size $h$ and the shift parameter $t_i^2$. The projected eigenvalue problem to be solved is of dimension $r+1$, amounting to an overall computational effort of order $\mathcal{O}(rN^2)$. Thanks to its rapid convergence, the costs are diminished significantly compared to the complexity $\mathcal{O}(N^3)$ the evaluation of the truth solution requires. 
 	
	We highlight that the assembly of the reduced space is entirely independent of the fractional order and thus can be computed efficiently in the course of an offline-online decomposition. In this case, evaluations of $u(s)$ for several values of $s\in(0,1)$ can be performed at neglectable extra costs.
	 
	The final component required to conclude the description of our algorithm is the specification of snapshots in \eqref{Def:ReducedSpace}. Traditional reduced basis procedures perform this choice on the basis of a weak greedy algorithm, where each snapshot is determined as minimizer of a suitable objective function, see \cite{RBMWeakGreedy}. Apart from a single offline computational investment, the choice we propose in the following chapter satisfies optimality properties without the need of any further computations.

	\section{Analysis}\label{SecAnalysis}
	The goal of this chapter is to illuminate the precise choice of snapshots $t_1,...,t_r$ from \eqref{Def:ReducedSpace} to achieve optimal convergence rates. Their selection has been analyzed in \cite{RBMForFracDiff} and relies on a special type of rational approximation of the function $(1+t^2\lambda^2)^{-1}$, with $\lambda^2$ residing on the spectral interval of the discrete operator. Thereupon, we state the following definition, see also \cite{ZolotarevCollectedWorks} and \cite{ZolotarevProbGonchar}.
	\begin{Def}\label{Def:ZolotarevPoints}
		Let $\delta\in(0,1)$. For each $r\in\N$ we define the Zolotar\"ev points $\mathcal{Z}_1,...,\mathcal{Z}_r$ on $[\delta,1]$ by
	 	\begin{align*}
	 		\mathcal{Z}_j := \DeltaAmpl\left(\frac{2(r-j)+1}{2r}\mathcal{K}(\delta'),\delta'\right), \rlap{\qquad$j = 1,...,r,$}
	 	\end{align*}
	 	where $\DeltaAmpl(\theta,k)$ denotes the Jacobi elliptic function, $\mathcal{K}(k)$ the elliptic integral of first kind with elliptic modulus $k$, and $\delta':= \sqrt{1-\delta^2}$. For any arbitrary interval $[a,b]\subset\R^+$, the transformed Zolotar\"ev points are defined by
	 	\begin{align*}
		 	\widehat{\mathcal{Z}}_j := b\mathcal{Z}_j, \rlap{\qquad$j = 1,...,r$,}
	 	\end{align*}
	 	where $\mathcal{Z}_1,...,\mathcal{Z}_r$ denote to the Zolotar\"ev points on $\left[\frac{a}{b},1\right]$.
	\end{Def}
	We refer to \cite[Section 16 \& 17]{HandbookOfMathFunc} for a concise review of Jacobi elliptic functions and elliptic integrals . The Zolotar\"ev points are roughly geometrically distributed and accumulate at the left endpoint of the interval. Based on the spectrum $\sigma(\Op_H)$ of the discrete operator, we are now in position to specify the optimal choice of snapshots.
	\begin{Def}\label{Def:ZolotarevSpace}
	 	A reduced space $\Vr = \Span\{\minim{N}(t_0),...,\minim{N}(t_r)\}\subset \Vh$ is called Zolotar\"ev space, if and only if
	 	there exists a positive spectral interval $\sigma := [\lambda_L^{2},\lambda_U^2]\supset \sigma(\Op_H)$, such that the squared snapshots $t_1^2,...,t_r^2$ coincide with the transformed Zolotar\"ev points on $[\lambda_U^{-2}, \lambda_L^{-2}]$. We call $\kappa := \nicefrac{\lambda_U^2}{\lambda_L^2}$ the estimated condition number.
	\end{Def}
	With this at hand, we present the theoretical key result of this paper.
	\begin{Th}[Exponential convergence of the dual reduced basis algorithm]\label{Th:Core}
		Let $f\in \Vh$ and $\Vr\subset \Vh$ a Zolotar\"ev space with $\sigma = [\lambda_L^2,\lambda_U^2]$ and $\delta = \nicefrac{\lambda_L^2}{\lambda_U^2}$. Then there exists a constant $C^*\in\R^+$, such that
		\begin{align}\label{SecAnalysis:NormError}
			0\leq \IntNormRBDual{f}{H}{-s}^2 - \Norm{f}{H^{-s}}^2  \preceq e^{-2C^*r}\Norm{f}{0}^2.
		\end{align}
		Moreover, there holds
		\begin{align}\label{SecConvAna:L2ConvOp}
			\|\RBSolutionDual{s}-u(s)\|_{0} \preceq e^{-C^*r}
			\begin{cases}
				\Norm{f}{0}, \quad &s>\frac{1}{2}, \\
				\Norm{f}{1}, \quad &s\leq\frac{1}{2}.
			\end{cases}
		\end{align}
		The constant $C^*$ only depends on the estimated condition number $\kappa$ and satisfies
		\begin{align}\label{SecAnalysis:AsymptoticC*}
			C^*(\kappa) = \mathcal{O}\left(\frac{1}{\ln\kappa}\right),\quad\text{as }\kappa\to\infty.
		\end{align}
		Its precise value can be specified by
		\begin{align*}
			C^* = \frac{\pi \mathcal{K}(\mu_1)}{4\mathcal{K}(\mu)},\qquad\quad \mu := \left(\frac{1-\sqrt{\delta}}{1+\sqrt{\delta}}\right)^2,\qquad\quad \mu_1 := \sqrt{1-\mu^2}.
		\end{align*}
	\end{Th}
	\begin{proof}
		The first part of the theorem, namely \eqref{SecAnalysis:NormError}, is a direct consequence of \cite[Theorem 5.4]{RBMForFracDiff} and the fact that solutions to \eqref{SecRBM:EulerLagrangeLSE} indeed coincide with minimizers of the dual K-functional up to identification, as shown in Lemma \ref{Lm:MinimDecomp} and \ref{Lm:EqualMinim}. 
		
		To confirm \eqref{SecConvAna:L2ConvOp}, it suffices to prove convergence with respect to $\IntOpRBDual{K}{-s}$, the rest follows from Theorem \ref{Th:DualvsNegative}. For the ease of use we define the operators
		\begin{align*}
			\Op_{K}^{-s} := \RieszVh\Op_{K_*^{1-s}} := \RieszVh\IntOpGeneral{K}{s}{\Vh'}{\Norm{\cdot}{-1}}{\Vh}{\Norm{\cdot}{0}}.
		\end{align*}
		Let $K_r^2(t;f)$ and $K^2(t;f)$ denote the K-functionals of $((\Vr',\Norm{\cdot}{-1}),(\Vr,\Norm{\cdot}{0}))$ and $((\Vh',\Norm{\cdot}{-1}),(\Vh,\Norm{\cdot}{0}))$, respectively, and $\minim{r}(t)\in\Vr$ the minimizer of $K_r^2(t;f)$. Due to Theorem \ref{Th:IntRepr} and Cauchy-Schwarz inequality, we observe
		\begin{align*}
			\left|\scp{w}{\Op_{K_{*,r}^{1-s}}(f)}{-1}-\scp{w}{\Op_{K_*^{1-s}}f}{-1}\right|&\leq \Norm{w}{-1}\int_0^\infty t^{-2s-1}\Norm{\minim{r}(t) - \minim{N}(t)}{-1}\,dt \\
			&\leq \Norm{w}{-1}\int_0^\infty t^{-2s-1}\left(K_r^2(t;f) - K^2(t;f)\right)dt,
		\end{align*}
		where the last inequality follows from \cite[Corollary 5.9]{RBMForFracDiff}. Theorem 5.24 in \cite{RBMForFracDiff} reveals for all $w\in\Vh$
		\begin{align*}
			\left|\scp{w}{\Op_{K_{*,r}^{1-s}}(f)}{-1}-\scp{w}{\Op_{K_*^{1-s}}f}{-1}\right| \preceq e^{-C^*r}\|w\|_{-1} \Norm{f}{1} \preceq e^{-C^*r}\|w\|_{0}\|f\|_{1},
		\end{align*}
		where $\|f\|_{1}$ can be replaced with $\Norm{f}{0}$ as $1-s<\frac{1}{2}$, i.e., as $s>\frac{1}{2}$. This yields
		\begin{align*}
			\|\IntOpRBDual{K}{-s}(f) - \Op_{K}^{-s}f\|_{0} &= \sup\limits_{w\in \Vh\setminus\{0\}}\frac{|\scp{w}{\IntOpRBDual{K}{-s}(f) - \Op_{K}^{-s}f}{0}|}{\|w\|_{0}} \\
			&= \sup\limits_{w\in \Vh\setminus\{0\}}\frac{\left|\scp{w}{\RieszVh\IntOpRBDual{K^{1-s}}{}(f) - \RieszVh\Op_{K_*^{1-s}}f}{0}\right|}{\|w\|_{0}} \\
			&= \sup\limits_{w\in \Vh\setminus\{0\}}\frac{|\scp{w}{\IntOpRBDual{K^{1-s}}{}(f)}{-1} - \scp{w}{\Op_{K_*^{1-s}}f}{-1}|}{\|w\|_{0}} \\
			&\preceq e^{-C^*r}
			\begin{cases}
				\Norm{f}{0}, \quad s>\frac{1}{2},\\
				\Norm{f}{1}, \quad s\leq\frac{1}{2},			
			\end{cases}
		\end{align*}
		concluding the proof.
 	\end{proof}
 	\begin{Rem}
 		Even though \eqref{SecAnalysis:AsymptoticC*} stems from an asymptotic identity, it is observed experimentally that this characterization appears to be rather accurate already for small values of $\kappa$. In practice, one is usually interested in the number of solves required to guarantee a prescribed precision $\varepsilon>0$. This can be achieved by means of equation \eqref{SecAnalysis:AsymptoticC*}, such that
 		\begin{align*}
 		r = \mathcal{O}(\ln\varepsilon\ln(\kappa^{-1})).
 		\end{align*}
 	\end{Rem}

	\section{Numerical examples}\label{SecNumericalEx}
	In this section, we conduct an empirical comparison of the reduced basis norms and operators to affirm their predicted convergence rates from Theorem \ref{Th:Core}. For concreteness, we set $\Vnull := (L_2(\Omega),\Norm{\cdot}{L_2})$ and $\Veins := (H_0^1(\Omega),\Norm{\nabla\cdot}{L_2})$, $\Omega\subset\R^2$, in all our experiments to study the model problem \eqref{SecInt:ModelProblem}. We further choose $\Vh\subset H_0^1(\Omega)$ to be a finite element space of polynomial order $p$ on a quasi-uniform, triangular mesh $\mathcal{T}_h$ of mesh size $h$. All numerical examples were implemented within the finite element library Netgen/NGSolve\footnote{www.ngsolve.org}, see \cite{Netgen} and \cite{NGSolve}. Implementation of the Zolotar\"ev points is performed by means of the special function library from \texttt{Scipy}\footnote{https://docs.scipy.org/doc/scipy/reference/special.html}.
	\begin{Ex}\label{Ex:1}
		We are concerned with the approximation of the dual interpolation norms arising from $\Vnull$ and $\Veins$ on the unit square $\Omega = (0,1)^2$. To make matters precise, we set $p = 2$ and $h = 0.01$. On $\Omega$, the smallest eigenvalue of $(\Vnull,\Veins)$ is known explicitly, namely $2\pi^2$, and thus can be utilized as lower bound for $\sigma(\Op_H)$. We set $\lambda_L^2 = 2\pi^2$ and $\lambda_U^2 = 1721511$, where the latter stems from an approximation obtained by power iteration. For the ease of use, we define
		\begin{align*}
			\textrm{e}_*(r,s,h) := \IntNormRBDual{f}{H}{-s} - \Norm{f}{H^{-s}},\qquad \textrm{e}(r,s,h) := \left|\IntNormRB{f}{H}{-s} - \Norm{f}{H^{-s}}\right|.
		\end{align*}
		Note that $\textrm{e}_*(r,s,h) \geq 0$. The error of the norm approximations is illustrated in Figure \ref{Fig:ConvergenceNorms}, where $f$ is chosen to be the $L_2$-orthonormal projection of the constant $1-$function onto $\Vh$. Exponential convergence is observed in all cases. The speed at which the error decreases is faster than the predicted worst-case scenario from Theorem \ref{Th:Core}, where $C^*\approx 0.39$.
		\begin{figure}[!h]
			\includegraphics[width=0.58\textwidth]{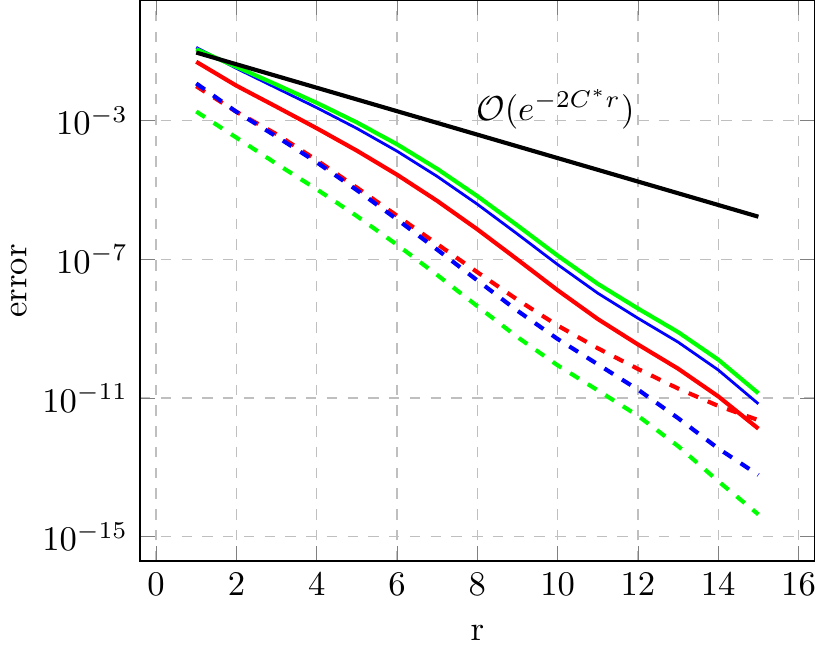}
			\caption{Error $\textrm{e}_*(r,s,0.01)$ (dashed) and $\textrm{e}(r,s,0.01)$ (solid) for $s = 0.1$ (red), $s = 0.5$ (blue), and $s = 0.9$ (green).}\label{Fig:ConvergenceNorms}
		\end{figure}
	\end{Ex}
	\begin{Ex}
		We set $p = 1$ and $h = 0.04$ to study numerically the accuracy of the reduced basis surrogates on the L-shape domain $\Omega := (0,1)^2\setminus([0,0.5]\times[0.5,1])$. Based on numerical approximations of the extremal eigenvalues, we choose $\lambda_L^2 = 18$ and $\lambda_U^2 =  18083$. Moreover, we introduce 
		\begin{align*}
			{E}_*(r,s,h) := \Norm{\RBSolutionDual{s} - u(s)}{L_2},\qquad E(r,s,h) := \Norm{u_r(s) - u(s)}{L_2},
		\end{align*}
		to report the discrepancy between the truth solution $u(s)\in\Vh$ and its reduced basis approximations for various choices of $s$ and a randomly chosen $f\in\Vh$ in Figure \ref{Fig:ConvergenceOperators}. As predicted by Theorem \ref{Th:Core}, exponential convergence rates of order $C^* \approx 0.63$ are observed for the dual reduced basis approximations, irrespectively of the fractional order, incorporating the very same reduced space. Moreover, the example indicates that $u_r(s)$ enjoys to the same convergence properties as its predecessor. The extrapolation-based reduced basis approximation outperforms $u_r^*(s)$ as $s\leq \frac{1}{2}$ and appears to be less prone to small values of $s$.	
		\begin{figure}[ht]
			\begin{minipage}[t]{0.485\linewidth}
				\centering							
				\includegraphics[width=\textwidth]{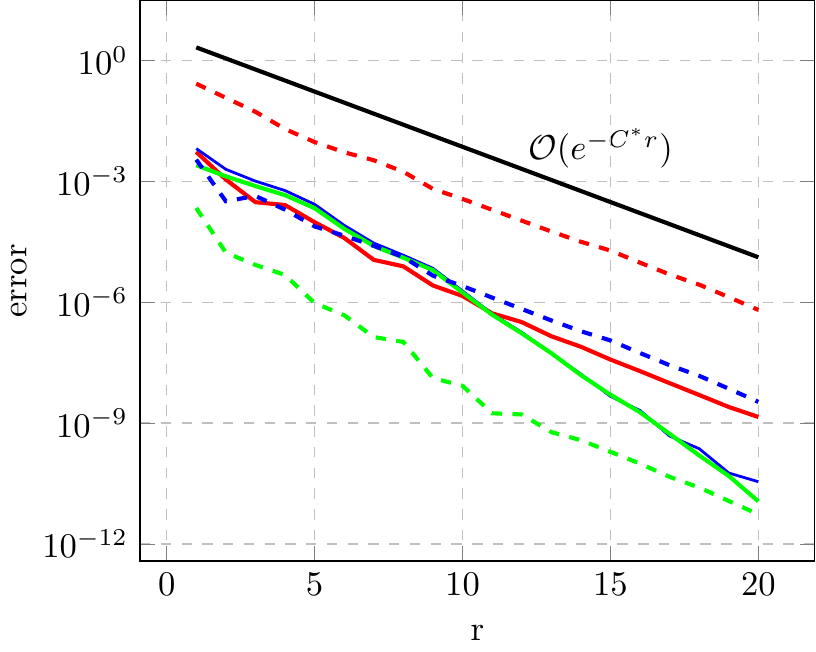}
				\captionsetup{width = \linewidth}
				\caption{Error $E_*(r,s,0.04)$ (dashed) and $E(r,s,0.04)$ (solid) for $s = 0.1$ (red), $s = 0.5$ (blue), and $s = 0.9$ (green).}\label{Fig:ConvergenceOperators}
			\end{minipage}
			\hspace{0.1cm}
			\begin{minipage}[t]{0.485\linewidth}
				\centering
				\includegraphics[width=\textwidth]{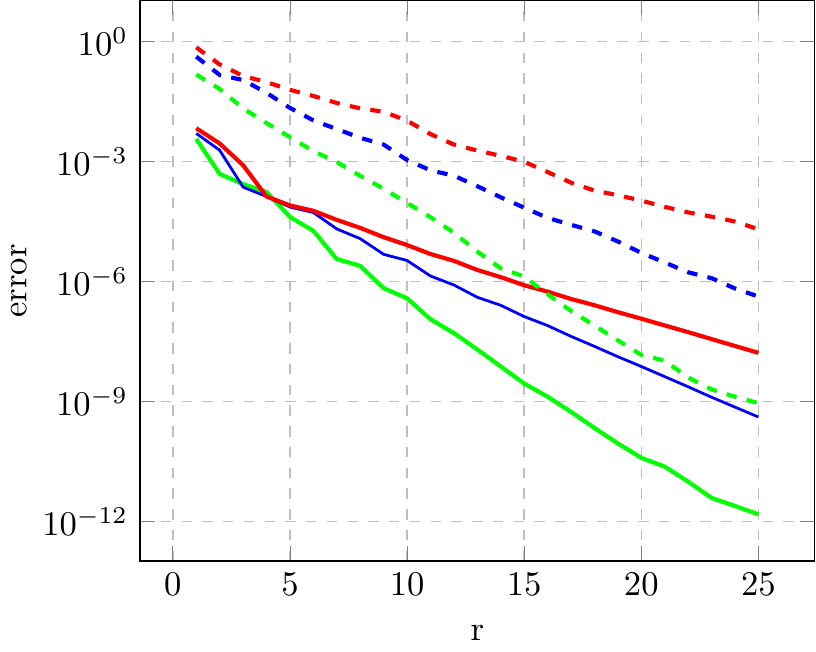}
				\captionsetup{width = \linewidth}
				\caption{Error $E_*(r,0.1,h)$ (dashed) and $E(r,0.1,h)$ (solid) for $h = 1.3\cdot 2^{-k}$, $k=4,6,8$, in green, blue, and red, respectively.}\label{Fig:Connditon}
			\end{minipage}
		\end{figure}
	
		The rate of convergence deteriorates as the problem becomes ill-conditioned, i.e., as $h\to 0$. Figure \ref{Fig:Connditon} mirrors the impact of decreasing mesh size on the performance of our algorithms for a randomly chosen right-hand side $f\in\Vh$ and $s = 0.1$. In accordance with the theory, a smooth transition between $\mathcal{O}(e^{-0.82r})$ and $\mathcal{O}(e^{-0.39r})$ is observed. The logarithmical dependency on the estimated condition number seemingly also applies to $u_r(s)$. As in the previous example, $u_r(s)$ is superior to its dual counterpart despite the reduced computational effort.
	\end{Ex}

	\section{Appendix}
	\begin{proof}[Proof of Theorem \ref{Th:IntRepr}]
		It suffices to show that
		\begin{align}\label{SecIntro:KFunc-Repr}
			\K_{(\Vnull,\Veins)}^2(t;u) = \Norm{u}{0}^2 - \scp{u}{\minim{}(t)}{0}.
		\end{align}
		There holds
		\begin{align*}
			\Norm{u-\minim{}(t)}{0}^2 = \Norm{u}{0}^2 - 2\scp{u}{\minim{}(t)}{0} + \Norm{\minim{}(t)}{0}^2. 
		\end{align*}
		Let $u_k := \scp{u}{\varphi_k}{0}$ to deduce from Lemma \ref{Lm:MinimDecomp}
		\begin{align*}
			t^2\Norm{\minim{}(t)}{1}^2 = \Sum{k=1}{\infty}\frac{t^2\lambda_k^2u_k^2}{(1+t^2\lambda_k^2)^2} = \Sum{k=1}{\infty}\frac{u_k^2}{1+t^2\lambda_k^2} - \frac{u_k^2}{(1+t^2\lambda_k^2)^2} = \scp{u}{\minim{}(t)}{0} - \Norm{\minim{}(t)}{0}^2, 
		\end{align*}
		which proves \eqref{SecIntro:KFunc-Repr} and concludes the proof.
	\end{proof}
	\begin{proof}[Proof of Theorem \ref{Thm:DualEigenpairs}]
		One observes that for any $F\in\Vnull$ 
			\begin{align*}
				\scp{F}{\varphi_k}{} = \scp{\Riesz F}{\varphi_k}{1}
				&= \lambda_k^2\scp{\Riesz F}{\varphi_k}{0}	= \lambda_k^{2}\scp{\Phi_k}{\Riesz F}{} = \lambda_k^{2}\scp{\Phi_k}{F}{-1},
			\end{align*}
		from which we conclude that $(\lambda_k\Phi_k)_{k=1}^\infty$ is a $\Hilbert{-1}$-orthonormal system of eigenfunctions. Since
		\begin{align*}
			0 = \scp{F}{\Phi_k}{-1} = \lambda_k^{-2}\scp{F}{\varphi_k}{}
		\end{align*}
		for all $k\in\N$ implies that $F = 0$, it is also a basis. This proves the claim.
	\end{proof}

	\begin{proof}[Proof of Theorem \ref{Th:DualvsNegative}]
		Due to
		\begin{align*}
			\scp{F}{\Phi_k}{-1} = \lambda_k^{-2}\scp{F}{\varphi_k}{}
		\end{align*}
		and Theorem \ref{Thm:DualEigenpairs}, there holds
		\begin{align*}
			\IntNormDual{F}{H}{1-s}^2 = \Sum{k=1}{\infty}\lambda_k^{2-2s}\scp{F}{\Psi_k}{-1}^2 =  \Sum{k=1}{\infty}\lambda_k^{-2s}\scp{F}{\varphi_k}{}^2
			= \IntNorm{F}{H}{-s}^2,
		\end{align*}
		proving the first equality in \eqref{SecIntro:NormEqualityNew}. The second one follows by means of \eqref{SecIntro:NormEquality}. Furthermore, one observes
		\begin{align*}
			\Riesz\IntOpDual{H}{1-s}F &= \Riesz\Sum{k=1}{\infty}\lambda_k^{2-2s}\scp{F}{\Psi_k}{-1}\Psi_k = \Riesz\Sum{k=1}{\infty}\lambda_k^{-2s}\scp{F}{\psi_k}{}\Psi_k \\ 
			&= \Sum{k=1}{\infty}\lambda_k^{2-2s}\scp{F}{\varphi_k}{}\Riesz\Phi_k = \Sum{k=1}{\infty}\lambda_k^{-2s}\scp{F}{\varphi_k}{}\varphi_k = \IntOp{H}{-s}F,
		\end{align*}
		confirming the first equality in \eqref{SecIntro:EqualOperators}. The latter is a consequence of \eqref{SecIntro:IntOpEquality}. The remainder follows as $\IntOp{H}{-s}F = \IntOp{H}{s}^{-1}f$ for $F\in\Vnull$.
	\end{proof}
	
	\section*{Acknowledgements}
	The authors acknowledge support from the Austrian Science Fund (FWF) through grant number F 65 and W1245.

	\bibliography{RBM2Bibliography}{}
	\bibliographystyle{amsplain}
	
\end{document}